\documentclass[12pt,a4paper]{amsart}
\usepackage[width=0.8\paperwidth,height=0.85\paperheight,twoside=false,includehead,includefoot]{geometry}
\allowdisplaybreaks[1]

\usepackage{tikz}
\usepackage[all]{xy}

\newtheorem{thm}{Theorem}[section]
\newtheorem{lem}[thm]{Lemma}
\newtheorem{prop}[thm]{Proposition}
\newtheorem{cor}[thm]{Corollary}

\theoremstyle{definition}

\theoremstyle{remark}
\newtheorem{rem}[thm]{Remark}

\DeclareMathOperator{\Aut}{Aut}
\DeclareMathOperator{\End}{End}

\DeclareMathOperator{\Imm}{Im}
\DeclareMathOperator{\Aug}{Aug}

\newcommand{\I}{\mathbb{I}}
\newcommand{\id}{\operatorname{id}}

\newcommand{\harub}{\mathbin{\overline{\ast}}}

\begin{document}

\title[Algebraic aspects of rooted tree maps]{Algebraic aspects of rooted tree maps}

\author{Hideki Murahara}
\address[Hideki Murahara]{Nakamura Gakuen University Graduate School, 5-7-1, Befu, Jonan-ku,
Fukuoka-city, 814-0198, Japan}
\email{hmurahara@nakamura-u.ac.jp}

\author{Tatsushi Tanaka}
\address[Tatsushi Tanaka]{Department of Mathematics, Faculty of Science, Kyoto Sangyo University, Motoyama, Kamigamo,
Kita-ku, Kyoto-city 603-8555, Japan}
\email{t.tanaka@cc.kyoto-su.ac.jp}


\begin{abstract}
Based on the Connes--Kreimer Hopf algebra of rooted trees, the rooted tree maps are defined as linear maps on noncommutative polynomial algebra in two indeterminates. It is known that they induce a large class of linear relations for multiple zeta values. 
In this paper, we investigate some basic algebraic properties of rooted tree maps by relating to the harmonic algebra. 
We also characterize the antipode maps as the conjugation by the special map $\tau$. 

\end{abstract}

\maketitle

\section{Introduction} 
Let $\mathcal{H}$ be the Connes--Kreimer Hopf algebra of rooted trees introduced in \cite{CK98}. 
By assigning to $f \in\mathcal{H}$, the rooted tree map $\tilde{f}$ is introduced in \cite{Tan19} as an element in $\End(\mathcal{A})$, 
where $\mathcal{A}$ denotes the noncommutative polynomial algebra $\mathbb{Q}\langle x,y\rangle$. 
It is known that the rooted tree maps induce a large class of linear relations for multiple zeta values. 
In \cite{BT18,BT20}, 
we find some results in algebraic properties of rooted tree maps to make some applications to multiple zeta values clear. 
In this paper, we establish new algebraic formulas for rooted tree maps in the harmonic algebra. 

Our first theorem gives an explicit formula for rooted tree maps by using the product $\diamond$ (defined in \cite{HMO19}), which is a variation of the so-called harmonic product on $\mathcal{A}$. 
To state the formula, we need a polynomial $F_f\in\mathcal{A}$ determined by $f\in\mathcal{H}$, which we define in Section 3.   
\begin{thm} \label{main1}
 For any $f \in\mathcal{H}$ and $w\in\mathcal{A}$, we have
 \[
  \tilde{f}(wx)=(F_f\diamond w)x.
 \]
\end{thm}

Our second theorem shows that another polynomial $G_{f}\in\mathcal{A}$ determined by $f\in\mathcal{H}$ gives a similar formula for $\widetilde{S(f)}\in\End(\mathcal{A})$, where $S$ denotes the antipode of $f$ (see Section 4 for the precise definition of $G_f$). 
\begin{thm} \label{main2}
 For any $f \in\mathcal{H}$ and $w\in\mathcal{A}$, we have
 \[
  \widetilde{S(f)}(wx)=(G_f\diamond w)x.
 \]
\end{thm}
\noindent
By theorems \ref{main1} and \ref{main2}, we have $(G_f\diamond w)x=\widetilde{S(f)}(wx)=(F_{S(f)}\diamond w)x$ for $w\in\mathcal{A}$. 
Thus we obtain
\begin{cor} \label{cor}
 For any $f \in\mathcal{H}$, we have
 \[
  G_{f}=F_{S(f)}.
 \]
\end{cor}
Let $\tau$ be the anti-automorphism on $\mathcal{A}$ with $\tau(x)=y$ and $\tau(y)=x$. 
This $\tau$ is an involution and known to give the dual index of multiple zeta values. 
Our third theorem asserts that $\tau$ plays another role to connect the rooted tree maps with their antipode maps. 
\begin{thm} \label{main3}
 For any $f \in\mathcal{H}$, we have
 \[
  \widetilde{S(f)}=\tau \tilde{f} \tau. 
 \]
\end{thm}

In Section 2, we discuss definitions and some basic properties of the Connes--Kreimer Hopf algebra of rooted trees, rooted tree maps, and harmonic products. 
Sections 3--5 are devoted to proofs of theorems \ref{main1}, \ref{main2}, and \ref{main3} in turn.

\section{Preliminaries}
\subsection{Connes--Kreimer Hopf algebra of rooted trees}
We review briefly the Connes--Kreimer Hopf algebra of rooted trees introduced in \cite{CK98}. 
A tree is a finite and connected graph without cycles and a rooted tree is a tree in which one vertex is designated as the root. 
We consider the rooted trees without plane structure, e.g., 
$\,\begin{xy}
   {(0,0) \ar @{{*}-{*}} (2,3)}, 
   {(2,3) \ar @{{*}-{*}} (4,0)},
   {(0,0) \ar @{{*}-{*}} (0,-3)}
\end{xy}\,\,
=\,\begin{xy}
   {(0,0) \ar @{{*}-{*}} (2,3)}, 
   {(2,3) \ar @{{*}-{*}} (4,0)},
   {(4,0) \ar @{{*}-{*}} (4,-3)},
\end{xy}\,\,$, 
where the topmost vertex represents the root.
A (rooted) forest is a finite collection of rooted trees $(t_1,\dots, t_n)$, which we denote the commutative product by $t_1\cdots t_n$. 
Then the Connes--Kreimer Hopf algebra of rooted trees $\mathcal{H}$ is the $\mathbb{Q}$-vector space freely generated by rooted forests with the commutative ring structure. 
We denote by $\I$ the empty forest, which is regarded as the neutral element in $\mathcal{H}$.

We define the linear map $B_+$ on $\mathcal{H}$ by grafting all roots of trees in a forest on a common new root and 
$B_+(\I)=  
 \,\,\begin{xy}
   {(0,0) \ar @{{*}-{*}} (0,0)}
 \end{xy}\,\,
$. 
We find that, for a rooted tree $t (\neq \I)$, there is a unique forest $f$ such that $t=B_+(f)$.
The coproduct $\Delta$ on $\mathcal{H}$ is defined by the following two rules. 
\begin{enumerate}
 \item $\Delta(t) = \I \otimes t + (B_{+} \otimes \id) \circ \Delta(f)$ if $t = B_+(f)$ is a tree,  
 \item $\Delta(f) = \Delta(g)\Delta(h)$ if $f=g h$ with $g,h \in \mathcal{H}$. 
\end{enumerate}
Note that components of the tensor product are reversely defined compared to those in \cite{CK98}. 
We denote by $S$ the antipode of $\mathcal{H}$. 
In the sequel, we often employ the Sweedler notation $\Delta(f)=\sum_{(f)} f' \otimes f''$. 

A subtree $t'$ of the rooted tree $t$ (denoted by $t'\subset t$) is a subgraph of $t$ that is connected and contains the root of $t$ (hence the empty tree $\I$ cannot be a subtree in our sense), and 
we denote by $t \setminus t'$ their subtraction.
For example, we have 
$t \setminus t'=
  \,\,\begin{xy}
   {(0,0) \ar @{{*}-{*}} (0,0)},
   {(4,-1.5) \ar @{{*}-{*}} (4,1.5)}
  \end{xy}\,\,$
if
$t=\,\begin{xy}
   {(0,0) \ar @{{*}-{*}} (2,3)}, 
   {(2,3) \ar @{{*}-{*}} (4,0)},
   {(4,0) \ar @{{*}-{*}} (4,-3)}
  \end{xy}\,\,$
and
$t'=  
  \,\begin{xy}
   {(0,0) \ar @{{*}-{*}} (0,0)}
  \end{xy}\,\,$.
\begin{prop}[{\cite{CK98}}] \label{cop}
 For a rooted tree $t$, we have 
 \begin{itemize}
   \item[(1)] $\displaystyle{\Delta(t)=\I \otimes t +\sum_{t' \subset t} t' \otimes (t \setminus t')}$,
   \item[(2)] $\displaystyle{S(t)+\sum_{t'\subset t} t' S(t \setminus t')=0}$.
 \end{itemize}
\end{prop}

\subsection{Rooted tree maps}
We here define rooted tree maps introduced in \cite{Tan19}.
For $u \in\mathcal{A}$, let $L_u$ and $R_u$ be $\mathbb{Q}$-linear maps on $\mathcal{A}$ defined by $L_u(w)=uw$ and $R_u(w)=wu \,\, (w\in \mathcal{A})$.
For $f\in \mathcal{H}$, we define the $\mathbb{Q}$-linear map $\tilde{f}\colon\mathcal{A}\to\mathcal{A}$, which we call the rooted tree map (RTM for short), recursively by 
\begin{enumerate}
 \item $\tilde{\I}=\id$, 
 \item $\tilde{f}(x)=yx$ and $\tilde{f}(y)=-yx$\, if 
 $f=
 \,\,\begin{xy}
  {(0,0) \ar @{{*}-{*}} (0,0)}
 \end{xy}\,\,
 $,
 \item $\tilde{t}(u)=L_y L_{x+2y}L_{y}^{-1} \tilde{f}(u)$ if $t = B_+(f)$ is a tree, 
 \item $\tilde{f}(u)=\tilde{g}(\tilde{h}(u))$\, if $f = gh$, 
 \item $\tilde{f}(uw)=\sum_{(f)} \tilde{f'}(u) \tilde{f''}(w)$\, for $\Delta(f)=\sum_{(f)} f' \otimes f''$, 
\end{enumerate}
where $w \in \mathcal{A}$ and $u \in \{x,y\}$.  
It is known that $\;{\widetilde{}}\;\colon \mathcal{H}\to\End(\mathcal{A})$ is an algebra homomorphism. 
We sometimes denote its image by $\widetilde{\mathcal{H}}$.
(Note that in this definition the order of the concatenation product on $\mathcal{A}$ is treated reversely compared to that in \cite{Tan19}. 
Since the coproduct $\Delta$ on $\mathcal{H}$ is also defined reversely as above, this definition makes sense.)

Let $z=x+y$. 
It is known that RTMs commute with each other and with $L_z$ and $R_z$. 
\begin{lem}[{\cite{Tan19}}] \label{z}
 For $f\in\mathcal{H}$ and $w\in\mathcal{A}$, we have $\tilde{f}(zw)=z\tilde{f}(w)$ and $\tilde{f}(wz)=\tilde{f}(w)z$.
\end{lem}

\subsection{Harmonic products}
Let $\mathcal{A}^1=\mathbb{Q}+y\mathcal{A}$ be the subalgebra of $\mathcal{A}$. 
We define the $\mathbb{Q}$-bilinear product $\ast$ on $\mathcal{A}^1$, which is called the harmonic product, by
\begin{align*} 
 \begin{split}
 w \ast 1&= 1\ast w=w, \\
 yx^{k_1-1}\cdots yx^{k_r-1} \ast yx^{l_1-1}\cdots yx^{l_s-1} 
 &= yx^{k_1-1} (yx^{k_2-1}\cdots yx^{k_r-1} \ast yx^{l_1-1}\cdots yx^{l_s-1}) \\
  &\quad +yx^{l_1-1} (yx^{k_1-1}\cdots yx^{k_r-1} \ast yx^{l_2-1}\cdots yx^{l_s-1}) \\
  &\quad +yx^{k_1+l_1-1} (yx^{k_2-1}\cdots yx^{k_r-1} \ast yx^{l_2-1}\cdots yx^{l_s-1}).
 \end{split}
\end{align*}
It is known that this product is commutative and associative, and has one of the product structures of ordinary multiple zeta values (see \cite{Hof97}). 
There are many results on the harmonic products. 
We here recall the following identity (see \cite[Proposition 6]{IKOO11} or \cite[Proposition 7.1]{Kaw09}). 
For $yx^{k_1-1}\cdots yx^{k_r-1}\in\mathcal{A}^1$, we have
\begin{align} \label{MA} 
 \sum_{i=0}^{r} (-1)^{i} yx^{k_1-1}\cdots yx^{k_i-1} \ast yx^{k_r-1}zx^{k_{r-1}-1}\cdots zx^{k_{i+1}-1} =0. 
\end{align}

Next, we define the $\mathbb{Q}$-bilinear product $\harub$ on $\mathcal{A}^1$ by
\begin{align*}
 w \harub 1&= 1\harub w=w, \\
 yx^{k_1-1}\cdots yx^{k_r-1} \harub yx^{l_1-1}\cdots yx^{l_s-1} 
 &= yx^{k_1-1} (yx^{k_2-1}\cdots yx^{k_r-1} \harub yx^{l_1-1}\cdots yx^{l_s-1}) \\
  &\quad +yx^{l_1-1} (yx^{k_1-1}\cdots yx^{k_r-1} \harub yx^{l_2-1}\cdots yx^{l_s-1}) \\
  &\quad -yx^{k_1+l_1-1} (yx^{k_2-1}\cdots yx^{k_r-1} \harub yx^{l_2-1}\cdots yx^{l_s-1}).
\end{align*}
Let $d_1$ be the automorphism on $\mathcal{A}$ given by $d_1(x)=x$ and $d_1(y)=z$.  
We define the $\mathbb{Q}$-linear map $d\colon\mathcal{A}^1\to \mathcal{A}^1$ by $d(1)=1$ and $d(yw)=yd_1(w)$. 
The map $d$ intermediates between the two products in the following sense. 
\begin{lem}[{\cite{Mun09}}] \label{muneta}
 For $w_1,w_2\in\mathcal{A}^1$, we have
 \[
  d(w_1 \harub w_2)
  =d(w_1) \ast d(w_2). 
 \]
\end{lem}

Lastly, following \cite{HMO19}, we define the product $\diamond$ on $\mathcal{A}$ by
\begin{align} \label{diamond}
 \begin{split}
 &\qquad\qquad w \diamond 1=1 \diamond w=w, \\
 &xw_1 \diamond xw_2 =x(w_1 \diamond xw_2) - x(yw_1 \diamond w_2), \\ 
 &xw_1 \diamond yw_2 =x(w_1 \diamond yw_2) + y(xw_1 \diamond w_2), \\ 
 &yw_1 \diamond xw_2 =y(w_1 \diamond xw_2) + x(yw_1 \diamond w_2), \\
 &yw_1 \diamond yw_2 =y(w_1 \diamond yw_2) - y(xw_1 \diamond w_2)
 \end{split}
\end{align}
for $w,w_1,w_2\in\mathcal{A}$ together with $\mathbb{Q}$-bilinearity.  
We find that the product $\diamond$ is associative and commutative. 
Let $\phi$ be the automorphism on $\mathcal{A}$ given by $\phi (x)=z$ and $\phi (y)=-y$.
We note that $\phi$ is an involution. 
The product $\diamond$ is thought of a kind of the harmonic product $\ast$ by virtue of $w_1 \diamond w_2= \phi (\phi(w_1) \ast \phi(w_2))$ for $w_1, w_2 \in \mathcal{A}$. 
\begin{lem}[{\cite[Proposition 2.3]{HMO19}}] \label{x+y}
 For $w_1,w_2\in\mathcal{A}$, we have
 \[ 
  zw_1\diamond w_2 = w_1\diamond zw_2 = z(w_1 \diamond w_2).
 \] 
\end{lem}

\begin{lem} \label{y}
 For $w_1,w_2\in\mathcal{A}$, we have 
 \begin{align*}
  w_1xw_2 \diamond y
  =(w_1 \diamond y) x w_2 +w_1 x (w_2 \diamond y).
 \end{align*}
\end{lem}
\begin{proof} 
 It is enough to show when $w_1$ is a word. 
 We prove the lemma by induction on $\deg(w_1)$. 
 When $\deg(w_1)=0$, we easily see the lemma holds. 
 Assume $\deg(w_1)\ge1$.
 If $w_1=zw_1'\,(w_1'\in\mathcal{A})$, by the induction hypothesis and Lemma \ref{x+y}, we have
 \begin{align*}
  \textrm{LHS} 
  =z(w_1'xw_2 \diamond y) 
  =z(w_1' \diamond y)xw_2 +zw_1'x(w_2 \diamond y) 
  =\textrm{RHS}. 
 \end{align*}
 If $w_1=xw_1'\,(w_1'\in\mathcal{A})$, by the induction hypothesis and \eqref{diamond}, we have
 \begin{align*}  
  \textrm{LHS} 
  =x(w_1'xw_2 \diamond y) +yw_1xw_2 
  =x(w_1' \diamond y)xw_2 +w_1x(w_2 \diamond y) +yw_1xw_2 
  =\textrm{RHS}. 
 \end{align*}
 This finishes the proof.
\end{proof}

\section{Proof of Theorem \ref{main1}}
In this section we prove Theorem \ref{main1}. 
For a forest $f$, we set a polynomial $F_f\in\mathcal{A}^1$ recursively by 
\begin{itemize}
 \item[(1)]  $F_{\I}=1$,
 \item[(2)] 
 $F_{\,\begin{xy}
   {(0,0) \ar @{{*}-{*}} (0,0)}
  \end{xy}\,\,}=y$, 
 \item[(3)] $F_t=L_yL_{x+2y}L^{-1}_y(F_f)$ if $t=B_+(f)$ is a tree and $f\ne\I$,
 \item[(4)] $F_f=F_g \diamond F_h$ if $f=gh$.
\end{itemize} 
The subscript of $F$ is extended linearly.
Put $L=L_yL_{x+2y}L^{-1}_y$.
To prove theorem \ref{main1}, next proposition plays a key role.  
\begin{prop} \label{key}
 For $w_1,w_2\in\mathcal{A}$ and $f\in \mathcal{H}$, we have 
 \begin{align*}
  w_1xw_2 \diamond F_f
  =\sum_{(f)} (F_{f'} \diamond w_1) x (F_{f''} \diamond w_2),
 \end{align*}
 where the sum on the right-hand side comes from the Sweedler notation of the coproduct $\Delta(f)$. 
\end{prop}
\begin{proof} 
 It is enough to show when $f$ is a forest. 
 We prove the proposition by induction on $\deg(f)$. 
 When $\deg(f)=1$, by Lemma \ref{y}, we find the proposition holds.  
 Assume $\deg(f)\ge2$.
 If $f=gh\,(g,h\ne\I)$, by the induction hypothesis and the multiplicativity of the coproduct, we have
 \begin{align*}
  w_1xw_2 \diamond F_f 
  &=w_1xw_2 \diamond (F_g \diamond F_h) \\
  &=(w_1xw_2 \diamond F_g) \diamond F_h \\
  &=\sum_{(g)} (F_{g'} \diamond w_1) x (F_{g''} \diamond w_2) \diamond F_h \\
  &=\sum_{(g)} \sum_{(h)} (F_{h'} \diamond (F_{g'} \diamond w_1)) x (F_{h''} \diamond (F_{g''} \diamond w_2)) \\
  &=\sum_{(g)} \sum_{(h)} ((F_{h'} \diamond F_{g'}) \diamond w_1) x ((F_{h''} \diamond F_{g''}) \diamond w_2) \\
  &=\sum_{(f)} (F_{f'} \diamond w_1) x (F_{f''} \diamond w_2).
 \end{align*}
 
 If $f$ is a tree (with $\deg(f)\ge2$), we have $F_f=L(F_g)$, where $f=B_+(g)$. 
 In this case, we prove the statement for a word $w_1$ by induction on $\deg(w_1)$.
 When $\deg(w_1)=0$, we have 
 \begin{align*}
  xw_2 \diamond F_f 
  &=xw_2 \diamond L(F_g) \\
  &=xw_2 \diamond yxL_{y}^{-1}F_g +xw_2 \diamond 2yF_g \\
  &=x(w_2 \diamond yxL_{y}^{-1}F_g) +y(xw_2 \diamond xL_{y}^{-1}F_g)
   +x(w_2 \diamond 2yF_g) +2y(xw_2 \diamond F_g ) \\
  &=y(xw_2 \diamond xL_{y}^{-1}F_g) +x(w_2 \diamond L(F_g)) +2y(xw_2 \diamond F_g ). 
 \end{align*}
 For the last term on the right-hand side, we have 
 \begin{align*}
  2y(xw_2 \diamond F_g )
  &=2\sum_{(g)} yF_{g'} x (F_{g''} \diamond w_2) \qquad \textrm{(by induction)} \\
  &=2yx(F_{g} \diamond w_2)
   +2\sum_{\substack{ (g) \\ g'\ne\I }} yF_{g'} x (F_{g''} \diamond w_2) \\
  &=2yx(F_{g} \diamond w_2)
   +\sum_{\substack{ (g) \\ g'\ne\I }} L(F_{g'}) x (F_{g''} \diamond w_2) 
   -\sum_{\substack{ (g) \\ g'\ne\I }} yxL_{y}^{-1}F_{g'} x (F_{g''} \diamond w_2).
 \end{align*}
 Then we find
 \begin{align*}
  xw_2 \diamond F_f 
  &=y(xw_2 \diamond xL_{y}^{-1}F_g) +x(w_2 \diamond L(F_g)) +2yx(F_{g} \diamond w_2) \\
  &\quad
   +\sum_{\substack{ (g) \\ g'\ne\I }} L(F_{g'}) x (F_{g''} \diamond w_2) 
   -\sum_{\substack{ (g) \\ g'\ne\I }} yxL_{y}^{-1}F_{g'} x (F_{g''} \diamond w_2).
 \end{align*}
 Since
 \begin{align*}
  x(w_2 \diamond L(F_g)) +yx(F_{g} \diamond w_2) 
   +\sum_{\substack{ (g) \\ g'\ne\I }} L(F_{g'}) x (F_{g''} \diamond w_2) 
  =\sum_{(f)} F_{f'} x (F_{f''} \diamond w_2) \\
  \textrm{(by Proposition \ref{cop} (1) or the definition of $\Delta$)}
 \end{align*}
 and
 \begin{align*}
  y(xw_2 \diamond xL_{y}^{-1}F_g)
  &=y(xL_{y}^{-1}F_g \diamond xw_2) \\
  &=yx(L_{y}^{-1}F_g \diamond xw_2) -yx(F_g \diamond w_2), 
 \end{align*}
 we have 
 \begin{align*}
  xw_2 \diamond F_f 
  &=\sum_{(f)} F_{f'} x (F_{f''} \diamond w_2)
   +yx(L_{y}^{-1}F_g \diamond xw_2) \\
  &\quad -\sum_{\substack{ (g) \\ g'\ne\I }} yxL_{y}^{-1}F_{g'} x (F_{g''} \diamond w_2).
 \end{align*}
 Here we see 
 \[
  L_{y}^{-1}F_g \diamond xw_2=\sum_{(g)} L_{y}^{-1}F_{g'} x (F_{g''} \diamond w_2)
 \]
 since 
 \begin{align*}
  y(L_{y}^{-1}F_g \diamond xw_2)
  &=yL_{y}^{-1}F_g \diamond xw_2 -x(w_2 \diamond F_g) \\
  &=F_g \diamond xw_2 -x(w_2 \diamond F_g) \\
  &=\sum_{ (g) } F_{g'} x (F_{g''}\diamond w_2)
   -x(w_2 \diamond F_g) \\
  &=\sum_{\substack{ (g) \\ g'\ne\I }} F_{g'} x (F_{g''}\diamond w_2).
 \end{align*} 
 Hence we get
 \[
  xw_2 \diamond F_f=\sum_{(f)} F_{f'} x (F_{f''} \diamond w_2). 
 \]
    
 Now we proceed to the case when $\deg(w_1)\ge1$. 
 If $w_1=zw_1'\,(w_1'\in\mathcal{A})$, we have
  \begin{align*}
   zw_1'xw_2 \diamond F_f
   &=z(w_1'xw_2 \diamond F_f) \\
   &=z\sum_{(f)} (F_{f'} \diamond w_1') x (F_{f''} \diamond w_2) \\   
   &=\sum_{(f)} (F_{f'} \diamond w_1) x (F_{f''} \diamond w_2)
  \end{align*}
 by the induction hypothesis. 
 %
 If $w_1=xw_1'\,(w_1'\in\mathcal{A})$, since we have already proved the identity in the case of $w_1=1$, we have
 \begin{align*}
  w_1xw_2 \diamond F_f
  &=\sum_{(f)} F_{f'} x (F_{f''} \diamond w_1' x w_2) \\
  &=\sum_{(f)} F_{f'} x \sum_{(f'')} (F_{f''_a} \diamond w_1') x (F_{f''_b} \diamond w_2),
 \end{align*}
 where we put $\Delta (f'')=\sum_{(f'')} f''_a \otimes f''_b$. 
 We also have
 \begin{align*}
  \sum_{(f)} (F_{f'} \diamond w_1) x (F_{f''} \diamond w_2)
  &=\sum_{(f)} (F_{f'} \diamond xw_1') x (F_{f''} \diamond w_2) \\
  &=\sum_{(f)} \sum_{(f')} F_{f'_a} x (F_{f'_b} \diamond w_1') x (F_{f''} \diamond w_2),
 \end{align*}
 where we put $\Delta (f')=\sum_{(f')} f'_a \otimes f'_b$. 
 By the coassociativity of $\Delta$, we find the result.
\end{proof}

\begin{proof}[Proof of Theorem \ref{main1}]
 We prove the theorem only for forests $f$ and words $w$ by induction on $\deg(f)$ and $\deg(w)$. 
 First, we prove the theorem when $\deg(f)=1$. 
 If $\deg(w)=0$, we easily find the result.  
 Suppose $\deg(w)\ge1$. 
 If $w=zw'\,(w'\in\mathcal{A})$, by Lemmas \ref{z} and \ref{x+y}, and the induction hypothesis, we have
 \[
  \textrm{LHS}
  =\tilde{f}(zw'x)
  =z\tilde{f}(w'x)
  =z(F_f\diamond w')x  
  =(F_f\diamond zw')x
  =\textrm{RHS}.
 \] 
 On the other hand, if $w=xw'\,(w'\in\mathcal{A})$, we have
 \[
  \textrm{LHS}
  =\tilde{f}(xw'x)
  =yxw'x +x\tilde{f}(w'x) 
  \]
  and
  \[
  \textrm{RHS}
  =(y \diamond xw')x  
  =yxw'x +x(y \diamond w')x.
 \] 
 By the induction hypothesis, we find the result. 
 
 Next, suppose $\deg(f)\ge2$. 
 If $f=gh\,(g,h\ne\I)$, we have
 \begin{align*}
  \tilde{f}(wx)
  =\tilde{g}\tilde{h}(wx) =\tilde{g}((F_h \diamond w)x) 
  =(F_g \diamond (F_h \diamond w))x 
  =((F_g \diamond F_h) \diamond w)x 
  =(F_f \diamond w)x.
 \end{align*}
 %
 Let $f$ be a rooted tree and put $f=B_+(g)$.
 When $\deg(w)=0$, we have
 \begin{align*}
  \tilde{f}(x)
  =(yxL_{y}^{-1} +2y) \tilde{g}(x) 
  =(yxL_{y}^{-1} +2y) F_g x 
  =F_f x.
 \end{align*}
 %
 Suppose $\deg(w)\ge1$. 
 If $w=zw'\,(w'\in\mathcal{A})$, we have
 \begin{align*}
  \tilde{f}(zw'x)
  =z\tilde{f}(w'x) 
  =z(F_f \diamond w')x 
  =(F_f \diamond zw')x
 \end{align*}
 by Lemmas \ref{z} and \ref{x+y}.
 If $w=xw'\,(w'\in\mathcal{A})$, we have
 \begin{align*}
  \tilde{f}(xw'x)
  =\sum_{(f)} \tilde{f'}(x) \tilde{f''}(w'x) 
  =\sum_{(f)} F_{f'} x (F_{f''} \diamond w')x 
 \end{align*}
 by the induction hypothesis. 
 By Proposition \ref{key}, we have
 \[
  (F_f \diamond xw')x
  =\sum_{(f)} F_{f'} x (F_{f''} \diamond w')x. 
 \]
 This completes the proof.
\end{proof}
\begin{rem}
 Although the commutativity of RTMs is guaranteed in general in \cite{Tan19}, it is hard to see at a glance. 
 The commutativity of RTMs on $\mathcal{A}x$ also follows from this theorem and the commutativity of the product $\diamond$. 
\end{rem}

\section{Proof of Theorem \ref{main2}}
Let $\mathcal{A}^1_{\ast}$ be the commutative $\mathbb{Q}$-algebra with the harmonic product $\ast$.  
We define the $\mathbb{Q}$-linear map $u\colon \mathcal{A}\to \mathcal{A}^\ast \otimes \mathcal{A}^\ast$ by $u(1)=1$ and sending a word $w=yx^{k_1-1}\cdots yx^{k_r-1}$ to
\[
 \sum_{i=0}^{r} (-1)^{i} yx^{k_1-1}\cdots yx^{k_i-1} \otimes yx^{k_r-1}zx^{k_{r-1}-1}\cdots zx^{k_{i+1}-1}.
\]
The notation $u_w$ is sometimes used instead of $u(w)$ for convenience. 
Let $\mathcal{B} \subset\mathcal{A}^1_{\ast}\otimes\mathcal{A}^1_{\ast}$ be the $\mathbb{Q}$-subalgebra 
algebraically generated by $u_{w}$'s.
The product of the tensor algebra is given componentwisely so that 
\begin{align*}
 &u(yx^{k_1-1}\cdots yx^{k_r-1}) \ast u(yx^{l_1-1}\cdots yx^{l_s-1}) \\
 &=\sum_{i=0}^{r} \sum_{j=0}^{s} (-1)^{i+j} (yx^{k_1-1}\cdots yx^{k_i-1} \ast yx^{l_1-1}\cdots yx^{l_j-1}) \\
  &\qquad\qquad \otimes (yx^{k_r-1}zx^{k_{r-1}-1}\cdots zx^{k_{i+1}-1} \ast yx^{l_s-1}zx^{l_{s-1}-1}\cdots zx^{l_{j+1}-1}).
\end{align*} 
Now we define the $\mathbb{Q}$-linear map $\rho\colon y\mathcal{A}\to y\mathcal{A}$ by setting $\rho(1)=1$ and $\rho=L_y\epsilon L_{y}^{-1}$, where $\epsilon$ is the anti-automorphism on $\mathcal{A}$ such that $\epsilon(x)=x$ and $\epsilon(y)=y$.
Note that $\rho(yx^{k_1-1}\cdots yx^{k_r-1})=yx^{k_r-1}\cdots yx^{k_1-1}$. 
Put $L'_{a}(w_1\otimes w_2)=yx^{a-1} w_1 \otimes w_2$ for $a\in\mathbb{Z}_{\ge1}$.  
\begin{lem} \label{uuu}
 For $w_1,w_2\in\mathcal{A}^1$, we have
\[
 u(w_1 \harub w_2)
 =u(w_1) \ast u(w_2). 
\]
\end{lem}
\begin{proof}
It is enough to show the lemma for $w_1=yx^{k_1-1}\cdots yx^{k_r-1}$ and $w_2=yx^{l_1-1} \cdots yx^{l_s-1}$.
The proof goes by induction on $r+s$. %
The lemma holds when $r+s\le1$ since $u(1)=1 \otimes 1$. 
Assume $r+s\ge2$. 
Note that 
\begin{align} \label{963}
 \begin{split}
  u(w)
  &=1 \otimes yx^{m_t-1}zx^{m_{t-1}-1}\cdots zx^{m_1-1} -L'_{m_1}u(yx^{m_2-1} \cdots yx^{m_t-1}) \\
  &=1 \otimes d\rho(w) -L'_{m_1}u(yx^{m_2-1} \cdots yx^{m_t-1})
 \end{split}
\end{align}
holds for $w=yx^{m_1-1}\cdots yx^{m_t-1}$. 
By definitions and the induction hypothesis, we have
\begin{align*}
 &u(w_1 \harub w_2) \\
 &=u(yx^{k_1-1} (yx^{k_2-1}\cdots yx^{k_r-1} \harub w_2) 
  +yx^{l_1-1} (w_1 \harub yx^{l_2-1} \cdots yx^{l_s-1}) \\
  &\qquad -yx^{k_1+l_1-1} (yx^{k_2-1}\cdots yx^{k_r-1} \harub yx^{l_2-1} \cdots yx^{l_s-1})) \\
 &=1 \otimes d\rho(w_1 \harub w_2) -L'_{k_1} (u(yx^{k_2-1} \cdots yx^{k_r-1}) \ast u(w_2)) \\
  &\quad +1 \otimes d\rho(w_1 \harub w_2) -L'_{l_1} (u(w_1) \ast u(yx^{l_2-1} \cdots yx^{l_s-1})) \\
  &\quad -1 \otimes d\rho(w_1 \harub w_2) +L'_{k_1+l_1} (u(yx^{k_2-1} \cdots yx^{k_r-1}) \ast u(yx^{l_2-1} \cdots yx^{l_s-1})) 
  \qquad \textrm{(by \eqref{963})} \\
 &=1 \otimes d\rho(w_1 \harub w_2) -L'_{k_1} (u(yx^{k_2-1} \cdots yx^{k_r-1}) \ast u(w_2)) 
    -L'_{l_1} (u(w_1) \ast u(yx^{l_2-1} \cdots yx^{l_s-1})) \\
  &\quad +L'_{k_1+l_1} (u(yx^{k_2-1} \cdots yx^{k_r-1}) \ast u(yx^{l_2-1} \cdots yx^{l_s-1}))
\end{align*}
and
\begin{align*}
 &u(w_1) \ast u(w_2) \\
 &=(1 \otimes d\rho(w_1) -L'_{k_1}u(yx^{k_2-1} \cdots yx^{k_r-1}))
  \ast (1 \otimes d\rho(w_2) -L'_{l_1}u(yx^{l_2-1} \cdots yx^{l_s-1})) \\
 &=1 \otimes (d\rho(w_1) \ast d\rho(w_2)) 
  -L'_{k_1} u(yx^{k_2-1} \cdots yx^{k_r-1}) \ast (1 \otimes d\rho(w_2)) \\
  &\quad -(1 \otimes d\rho(w_1)) \ast L'_{l_1} u(yx^{l_2-1} \cdots yx^{l_s-1}) 
   +L'_{k_1}u(yx^{k_2-1} \cdots yx^{k_r-1}) \ast L'_{l_1} u(yx^{l_2-1} \cdots yx^{l_s-1}). 
\end{align*}
Let us show that these two coincide. 
Because of Lemma \ref{muneta} and $\rho(w_1 \harub w_2)=\rho(w_1) \harub \rho(w_2)$, we have 
\[
 d\rho(w_1 \harub w_2)
 =d\rho(w_1) \ast d\rho(w_2). 
\]
Also we find that
\begin{align*}
 &-L'_{k_1} (u(yx^{k_2-1} \cdots yx^{k_r-1}) \ast u(w_2))
 +L'_{k_1} u(yx^{k_2-1} \cdots yx^{k_r-1}) \ast (1 \otimes d\rho(w_2)) \\
 &= -L'_{k_1} (u(yx^{k_2-1} \cdots yx^{k_r-1}) \ast u(w_2)
  -u(yx^{k_2-1} \cdots yx^{k_r-1}) \ast (1 \otimes d\rho(w_2))) \\
 &=-L'_{k_1} (u(yx^{k_2-1} \cdots yx^{k_r-1}) \ast L'_{l_1} u(yx^{l_2-1} \cdots yx^{l_s-1}))
\end{align*}
and
\begin{align*}
 &-L'_{l_1} (u(w_1) \ast u(yx^{l_2-1} \cdots yx^{l_s-1}))
 +(1 \otimes d\rho(w_1)) \ast L'_{l_1} u(yx^{l_2-1} \cdots yx^{l_s-1}) \\
 &= -L'_{l_1} (u(w_1) \ast u(yx^{l_2-1} \cdots yx^{l_s-1})
  -(1 \otimes d\rho(w_1)) \ast u(yx^{l_2-1} \cdots yx^{l_s-1})) \\
 &=-L'_{l_1} (L'_{k_1} u(yx^{k_2-1} \cdots yx^{k_r-1}) \ast u(yx^{l_2-1} \cdots yx^{l_s-1})). 
\end{align*}
Since
\begin{align*}
 &L'_{k_1+l_1} (u(yx^{k_2-1} \cdots yx^{k_r-1}) \ast u(yx^{l_2-1} \cdots yx^{l_s-1})) \\
  &-L'_{k_1}u(yx^{k_2-1} \cdots yx^{k_r-1}) \ast L'_{l_1} u(yx^{l_2-1} \cdots yx^{l_s-1}) \\
 &=L'_{k_1} (u(yx^{k_2-1} \cdots yx^{k_r-1}) \ast L'_{l_1} u(yx^{l_2-1} \cdots yx^{l_s-1})) \\
  &\quad +L'_{l_1} (L'_{k_1} u(yx^{k_2-1} \cdots yx^{k_r-1}) \ast u(yx^{l_2-1} \cdots yx^{l_s-1})), 
\end{align*}
we have the result. 
\end{proof}

Write $u_{w}=\sum_{i=0}^r u'_{w,i} \otimes u''_{w,i}=\sum_{w} u'_{w} \otimes u''_{w}$.
We define the $\mathbb{Q}$-linear maps $p,q\colon \mathcal{B}\to \mathcal{A}^\ast\otimes\mathcal{A}^\ast$ by
\begin{align*}
 p(u_{w_1}\ast\cdots\ast u_{w_r})
 &=\sum_{\substack{ u'_{w_1}\cdots u'_{w_r} \notin \mathbb{Q} \\ w_1,\dots,w_r }} 
  yxL_y^{-1} (u'_{w_1} \ast\cdots\ast u'_{w_r}) \otimes
   (u''_{w_1} \ast\cdots\ast u''_{w_r}) \\
  &\quad +1\otimes (d\rho(w_1)\ast\cdots\ast d\rho(w_1))x, \\
 q(u_{w_1}\ast\cdots\ast u_{w_r})
 &=\sum_{w_1,\dots,w_r} 
  y (u'_{w_1} \ast\cdots\ast u'_{w_r}) \otimes
   (u''_{w_1} \ast\cdots\ast u''_{w_r}) \\
  &\quad -1\otimes (d\rho(w_1)\ast\cdots\ast d\rho(w_1))z. 
\end{align*}
\begin{lem} \label{Im}
 We have $\Imm p, \Imm q\subset \mathcal{B}$.
\end{lem}
\begin{proof}
 From Lemma \ref{uuu}, we have
 \[
  u_{w_1}\ast\cdots\ast u_{w_r}
  =u(w_1 \harub \cdots \harub w_r).
 \]
 Thus, we need only to prove the lemma for the case $r=1$. 
 Since
 \begin{align*}
  p(u_w)
  &=1\otimes d\rho(w)x +\sum_{w} yxL_{y}^{-1} u'_w\otimes u''_w 
  =u(yx^{k_1} yx^{k_2-1}\cdots yx^{k_r-1}) \in \mathcal{B}, \\
  q(u_w)
  &=L'_y(u_w) -1\otimes d\rho(w)z 
  =u(y^2x^{k_1-1}yx^{k_2-1}\cdots yx^{k_r-1}) \in \mathcal{B},
 \end{align*}
 we obtain the result.
\end{proof}

For a forest $f$, we set a polynomial $G_f\in\mathcal{A}^1$ recursively by 
\begin{itemize}
 \item[(1)]  $G_{\I}=1$,
 \item[(2)] 
 $G_{\,\begin{xy}
   {(0,0) \ar @{{*}-{*}} (0,0)}
  \end{xy}\,\,}=-y$, 
 \item[(3)] $G_t=R_{2x+y}(G_f)$ if $t=B_+(f)$ is a tree and $f\ne\I$,
 \item[(4)] $G_f=G_g \diamond G_h$ if $f=gh$.
\end{itemize} 
The subscript of $G$ is extended linearly. 
The following lemma is immediate from Lemmas \ref{uuu} and \ref{Im}, and definitions. 
\begin{lem} \label{pqinA}
 Let $f$ be any forest with $f\ne\I$.
 If $\sum_{(f)} \phi(F_{f'}) \otimes \phi(G_{f''})\in \mathcal{B}$, we have
 \begin{align*}
  p\biggl(\sum_{(f)} \phi(F_{f'}) \otimes \phi(G_{f''}) \biggr)
  &=\sum_{\substack{ (f) \\ f'\ne\I }} yxL_{y}^{-1}\phi(F_{f'}) \otimes \phi(G_{f''})
   +\phi(F_{\I}) \otimes \phi(G_{f})x
   \in \mathcal{B}, \\
  q\biggl(\sum_{(f)} \phi(F_{f'}) \otimes \phi(G_{f''}) \biggr)
  &=\sum_{\substack{ (f) \\ f'\ne\I }} y\phi(F_{f'}) \otimes \phi(G_{f''})
  +y\phi(F_{\I}) \otimes \phi(G_{f})
  -\phi(F_{\I}) \otimes \phi(G_{f})z
  \in \mathcal{B}. 
 \end{align*}
\end{lem}
\begin{prop} \label{abc12345}
 For any forest $f\ne\I$, we have
 \[
  \sum_{(f)} \phi(F_{f'}) \otimes \phi(G_{f''}) \in \mathcal{B}.
 \] 
\end{prop}
\begin{proof}
 We prove the proposition by induction on $\deg(f)$. 
 When $\deg(f)=1$, we easily see the statement holds. 
 Assume $\deg(f)\ge2$. 
 If $f=gh\,(g,h\ne\I)$, 
 since $\phi(F_{g} \diamond F_{h}) =\phi(F_{g}) \ast \phi(F_{h})$,
 we have
 \begin{align*}
  \sum_{(f)} \phi(F_{f'}) \otimes \phi(G_{f''}) 
  &=\sum_{\substack{ (g) \\ (h) }} \phi(F_{g'} \diamond F_{h'}) \otimes \phi(F_{g''} \diamond F_{h''}) \\
  &=\sum_{(g)} \sum_{(h)} \left( \phi(F_{g'}) \otimes \phi(G_{g''}) \right) \ast
   \left( \phi(F_{h'}) \otimes \phi(G_{h''}) \right). 
 \end{align*}
 By the induction hypothesis, we find the result. 
 
 If $f$ is a tree, we put $f=B_{+}(g)$.
 Since
 \[
  \Delta(f)=\I\otimes f +(B_{+}\otimes \id)\Delta(g),
 \]
 we have
 \begin{align*}
  &\sum_{(f)} \phi(F_{f'}) \otimes \phi(G_{f''}) \\
  &=\phi(F_{\I}) \otimes \phi(G_{B_{+}(g)}) +\sum_{(g)} \phi(F_{B_{+}(g')}) \otimes \phi(G_{g''}) \\
  &=\phi(F_{\I}) \otimes \phi(G_{g}(2x+y)) +\sum_{\substack{ (g) \\ g'\ne\I }} \phi((yxL_{y}^{-1}+2y) F_{g'}) \otimes \phi(G_{g''}) 
   +\phi(yF_{\I}) \otimes \phi(G_{g}) \\
  &=\sum_{\substack{ (g) \\ g'\ne\I }} (yzL_{y}^{-1}-2y)\phi(F_{g'}) \otimes \phi(G_{g''}) 
   -y\phi(F_{\I}) \otimes \phi(G_{g}) +\phi(F_{\I}) \otimes \phi(G_{g})(x+z).
 \end{align*}
 Then we get
 \begin{align*}
  \sum_{(f)} \phi(F_{f'}) \otimes \phi(G_{f''})
  =(p-q) \biggl(\sum_{(g)} \phi(F_{g'}) \otimes \phi(G_{g''}) \biggr).
 \end{align*}
 By the induction hypothesis, we have $\sum_{(g)} \phi(F_{g'}) \otimes \phi(G_{g''})\in\mathcal{B}$.  
 Then, by Lemma \ref{pqinA}, we find the result. 
\end{proof}

Let $\Aug =\bigoplus_{n\ge1} \mathcal{H}_n$ be the augmentation ideal, where $\mathcal{H}_n$ is the degree $n$ homogeneous part of $\mathcal{H}$.
We define the $\mathbb{Q}$-linear map $M\colon\mathcal{A}^1_{\ast}\otimes\mathcal{A}^1_{\ast}\to\mathcal{A}^1_{\ast}$ by $M(w_1\otimes w_2)=w_1\ast w_2$. 
Note that $M(w)=0$ for $w\in\mathcal{B}$ by \eqref{MA} in Subsection 2.3.
\begin{prop} \label{key_main2}
 For any $f \in\Aug$, we have
 \begin{align*}
  \sum_{(f)} F_{f'} \diamond G_{f''}=0. 
 \end{align*}
\end{prop}
\begin{proof}
 We note that $\phi(w_1) \ast \phi(w_2)=\phi(w_1\diamond w_2)$ holds for $w_1,w_2\in\mathcal{A}$. 
 By Proposition \ref{abc12345}, we have 
 \[
  0=\sum_{(f)} M(\phi(F_{f'}) \otimes \phi(G_{f''}))
  =\sum_{(f)} \phi(F_{f'} \diamond G_{f''}).
 \] 
 Then we find the result. 
\end{proof}

\begin{proof}[Proof of Theorem \ref{main2}]
 We prove the theorem by induction on $\deg(f)$. 
 It is easy to see the theorem holds if $\deg(f)=1$. 
 Suppose $\deg(f)\ge2$. 
 If $f=gh\,(g,h\ne\I)$, we have
 \begin{align*}
  \widetilde{S(f)}(wx)
  &=\widetilde{S(gh)}(wx) =\widetilde{S(g)}((G_{h} \diamond w)x) \\
  &=(G_{g} \diamond (G_{h} \diamond w))x \\
  &=((G_{g} \diamond G_{h}) \diamond w)x \\
  &=(G_{f} \diamond w)x.
 \end{align*}
 %
 If $f=t$ is a tree, by Proposition \ref{key_main2}, Theorem \ref{main1}, and the induction hypothesis, we have
 \begin{align*} 
  (G_{t}\diamond w)x
  &=-\sum_{t'\subset t} ((F_{t'} \diamond G_{t \setminus t'})\diamond w)x \\
  &=-\sum_{t'\subset t} (F_{t'} \diamond (G_{t \setminus t'} \diamond w))x \\
  &=-\sum_{t'\subset t} \tilde{t'}((G_{t \setminus t'} \diamond w)x) \\
  &=-\sum_{t'\subset t} \tilde{t'} \widetilde{S(t \setminus t')} (wx). 
 \end{align*}
 Since $\widetilde{S(t)}+\sum_{t'\subset t} \tilde{t'} \widetilde{S(t \setminus t')}=0$ by Proposition \ref{cop} (2), we have
 \[ 
  (G_{t}\diamond w)x
  =\widetilde{S(t)}(wx). \qedhere
 \]
\end{proof}

\section{Proof of Theorem \ref{main3}} 
\begin{proof}[Proof of Theorem \ref{main3}]
 First, we prove the theorem when $w\in y\mathcal{A}x$. 
 Put $w=yw'x$.
 By Theorem \ref{main2} and Corollary \ref{cor}, we have
 \begin{align*}
  \widetilde{S(f)}(w)
  =(F_{S(f)} \diamond yw')x.
 \end{align*}
 We also have
 \begin{align*}
  \tau\tilde{f}\tau(w)
  &=\tau\tilde{f}(y\tau(w')x) \\
  &=\tau ((F_{f} \diamond y\tau(w'))x) \qquad \textrm{(by Theorem \ref{main1})} \\
  &=-\tau ((y\tau L_y^{-1}(F_{S(f)}) \diamond y\tau(w'))x) \qquad \textrm{(by Proposition \ref{aaaaaaa})} \\
  &=-y\tau (y\tau L_y^{-1}(F_{S(f)}) \diamond y\tau(w')) \\
  &=(F_{S(f)} \diamond yw')x \qquad \textrm{(by Lemma \ref{taurel})}.
 \end{align*}
 Thus we have
 \begin{align} \label{xyz}
  \widetilde{S(f)}(w)
  =\tau\tilde{f}\tau(w)
 \end{align}
 for $w\in y\mathcal{A}x$.
 
 Next, we prove the theorem when $w\in z\mathcal{A}x$ by induction on $\deg(w)$. 
 Put $w=zw'x$.
 Then, by Lemma \ref{z}, we have
 \begin{align*}
  \widetilde{S(f)}(w)
  &=z\widetilde{S(f)}(w'x), \\
  \tau\tilde{f}\tau(w)
  &=\tau\tilde{f}\tau(zw'x) 
  =z\tau\tilde{f}\tau(w'x).
 \end{align*}
 By \eqref{xyz} and the induction hypothesis, we have
 \begin{align} \label{xyyz}
  \widetilde{S(f)}(w'x)
  =\tau\tilde{f}\tau(w'x)
 \end{align} 
 for any $w'\in\mathcal{A}$, and hence the assertion. 
 
 Finally, we prove the theorem when $w\in \mathcal{A}z$ by induction on $\deg(w)$. 
 Put $w=w'z$.
 Then we have
 \begin{align*}
  \widetilde{S(f)}(w)
  &=(\widetilde{S(f)}(w'))z, \\
  \tau\tilde{f}\tau(w)
  &=\tau\tilde{f}\tau(w'z) 
  =(\tau\tilde{f}\tau(w'))z.
 \end{align*}
 By the induction hypothesis and \eqref{xyyz}, we have the assertion. 
 Therefore we have 
  $\widetilde{S(f)}(w) = \tau \tilde{f} \tau (w)$
 for any $w\in \mathcal{A}$. 
\end{proof}

\begin{prop} \label{aaaaaaa}
For $f\in\Aug$, we have
 \[
  F_f=-y\tau L_{y}^{-1} F_{S(f)}. 
 \]
\end{prop}
\begin{proof}
 It is sufficient to prove the proposition for forests $f$ by induction on $\deg(f)$. 
 Since 
 $F_{
  \,\begin{xy}
   {(0,0) \ar @{{*}-{*}} (0,0)}
  \end{xy}\,\,}
  =y$ and 
 $F_{S(
  \,\begin{xy}
   {(0,-0.3) \ar @{{*}-{*}} (0,-0.3)}
  \end{xy}\,\,)}
  =-y$, 
 the proposition hols for $\deg(f)=1$. 
 Suppose $\deg(f)\ge2$. 
 If $f=gh\,(g,h\ne\I)$, we have
 \begin{align*}
  F_f
  &=F_g \diamond F_h \\
  &=y\tau L_y^{-1}(G_g) \diamond y\tau L_y^{-1}(G_h) \qquad \textrm{(by induction and Corollary \ref{cor})} \\
  &=-R_x^{-1} \tau ((G_g \diamond G_h)x) \qquad \textrm{(by Lemma \ref{taurel} below)}
 \end{align*}
 and
 \begin{align*}
  y\tau L_y^{-1} G_f
  =y\tau L_y^{-1} (G_g \diamond G_h) 
  =R_x^{-1} \tau ((G_g \diamond G_h)x).
 \end{align*}
 Thus we have the result. 
 %
 If $f$ is a tree, put $f=B_+(g)$.
 Then we have
 \begin{align*} 
  F_f
  &=L(F_g) \\
  &=-L(y \tau L_{y}^{-1} G_g)  \qquad \textrm{(by induction and Corollary \ref{cor})} \\
  &=-y(x+2y) R_x^{-1} \tau (G_g)
 \end{align*}
 and
 \begin{align*} 
  -y\tau L_{y}^{-1} G_f 
  =-y\tau L_{y}^{-1} R_{2x+y} (G_g) 
  =-y(x+2y) R_x^{-1} \tau (G_g). 
 \end{align*}
 This finishes the proof. 
\end{proof}

Now we define $\sigma\in\Aut(\mathcal{A})$ such that $\sigma(x)=x$ and $\sigma(y)=-y$. 
By definitions, we have
\begin{align} \label{123456}
 -\phi R_{x}^{-1}\tau R_{x}\phi=d\rho\sigma. 
\end{align}
We find that $d, \sigma$, and $\rho$ are homomorphisms with respect to the harmonic product $\ast$, and $\rho$ commutes with $\sigma$. 
Hence the composition $d\rho\sigma$ is also a homomorphism with respect to the harmonic product $\ast$, and so is $-\phi R_x^{-1} \tau R_x \phi$ because of \eqref{123456}. 
This implies the composition $-R_x^{-1} \tau R_x$ is a homomorphism with respect to the product $\diamond$ (defined in Section 2) and hence we conclude the following lemma. 
\begin{lem} \label{taurel}
 For $w_1,w_2\in\mathcal{A}$, we have
 \[
  (yw_1 \diamond yw_2)x +y\tau(y\tau(w_1) \diamond y\tau(w_2))=0. 
 \]
\end{lem}
\begin{proof} 
 We have
 \begin{align*}
  yw_1\diamond yw_2 
  &=R_{x}^{-1}L_{y}(w_1x) \diamond R_{x}^{-1}L_{y}(w_2x) \\
  &=R_{x}^{-1}\tau R_{x}\tau(w_1x) \diamond R_{x}^{-1}\tau R_{x}\tau(w_2x) \\
  &=-R_{x}^{-1}\tau R_{x} (w_1x\diamond w_2x).  
 \end{align*}
 This gives the lemma. 
\end{proof}

\begin{rem}
 According to \cite{BT20}, for any $w\in y\mathcal{A}x$, there exists $\tilde{f}\in\widetilde{\mathcal{H}}$ such that $w=\tilde{f}(x)$. 
 Hence we have $(1-\tau)(w)=(1-\tau)(\tilde{f}(x))=(\tilde{f}+\tau \tilde{f}\tau)(x)=(\tilde{f}+\widetilde{S(f)})(x)$ due to Theorem \ref{main3}, which means each of the duality formulas for multiple zeta values also appears in this form in the context of RTMs. 
\end{rem}

\section*{Acknowledgement}
The second author is partially supported by JSPS KAKENHI Grant Number (C) 19K03434.


\end{document}